\def\confidentialstring{This is the accepted version of a paper
  published in \textit{Automation and Remote Control}, vol. 68, no. 9,
  2007, pp. 1527-1543,
\href{http://dx.doi.org/10.1134/S000511790709007X}{DOI:10.1134/S000511790709007X}
Russian version published in \textit{Avtomatika i Telemekhanika},
2007, No. 9, pp. 64-78; see \url{http://www.maik.ru}.
See also \href{http://www.reiszig.de/gunther/pubs/i07Convex.abs.html}{author's homepage}.
}

\documentclass[12pt]{article}
\usepackage{rotating, psfrag}
\usepackage{gunther2e}

\ifx
\hypersetup\undefined\def\href#1#2{\texttt{#2}}
\else
\hypersetup{
pdftitle={Convexity of reachable sets of nonlinear ordinary
  differential equations},
pdfauthor={Gunther Reissig},%
pdfsubject={\confidentialstring},
pdfkeywords={http://www.reiszig.de/gunther/pubs/i07Convex.abs.html}
}
\fi

\date{}

\title{Convexity of reachable sets of nonlinear ordinary differential equations}

\author{Gunther
\begin{picture}(0,0)\put(0,0){{\raisebox{4cm}{\makebox[0pt]{\begin{minipage}{\linewidth}\normalfont\scriptsize\confidentialstring
\end{minipage}}}}}
\end{picture}%
Rei\ss ig
}

\begin{document}

\maketitle

\begin{abstract}
We present a necessary and sufficient condition for the reachable set,
i.e., the set of states reachable from a ball of initial states at
some time, of an ordinary differential equation to be convex.
In particular, convexity is guaranteed if the ball of initial states
is sufficiently small, and we provide an upper bound on the radius of
that ball, which can be directly obtained from the right hand side of
the differential equation. In finite dimensions, our results cover the
case of ellipsoids of initial states.
A potential application of our results is inner and outer polyhedral
approximation of reachable sets, which becomes extremely simple
and almost universally applicable if these sets are known to be
convex. We demonstrate by means of an example that the balls of
initial states for which the latter property follows from our results
are large enough to be used in actual computations.
\end{abstract}

\section{Introduction}\label{s:intro}
Reachability problems play a central part in a wide range of control
related problems, including safety and liveness verification,
diagnosis, controller synthesis, optimization and others
\cite{Hsu87,BlankeKinnaertLunzeStaroswiecki03,%
ChutinanKrogh03,%
TomlinMitchellBayenOishi03,KurzhanskiVaraiya05,%
SingerBarton06b,%
JungeOsinga04}.
The vast majority of methods developed in that context compute
approximations of reachable sets in an intermediate step
\cite{HwangStipanovicTomlin03,ChutinanKrogh03,KurzhanskiVaraiya05},
which may simplify considerably if the reachable set
is known to be convex.
Consider, for example, an autonomous ordinary differential equation
$\dot x = f(x)$
with smooth flow
$\varphi\colon U\subseteq \mathbb{R}\times \mathbb{R}^n\to \mathbb{R}^n$,
a subset $\Omega \subseteq \mathbb{R}^n$ of initial states and some
$t_1\in\mathbb{R}$ with $\{t_1\} \times \Omega \subseteq U$, and assume
$\Omega$ is closed with nonempty interior and smooth boundary
$\boundary \Omega$.
Inner and outer polyhedral approximations to the reachable set
$\varphi(t_1,\Omega)$ from $\Omega$ at time $t_1$ are then computed
easily.
In particular, if $v$ is an outside normal to $\boundary \Omega$ at
$x \in \boundary \Omega$, an outside normal to the boundary
$\boundary \varphi(t_1,\Omega)$ at
$\varphi(t_1,x) \in \boundary \varphi(t_1,\Omega)$ can be obtained from
the solution of the adjoint to the variational equation
along $\varphi(\cdot,x)$ with initial value $v$ \cite{Hartman02}.
Thus, a convex reachable set may be efficiently approximated by inner
and outer polyhedra up to arbitrary precision,
see \ref{fig:1}.
\begin{figure}
\begin{center}
  \psfrag{x}[][]{$x$}
  \psfrag{v}[][]{$v$}
  \psfrag{O}[][]{$\Omega$}
  \psfrag{PTO}[][]{$\varphi(t_1,\Omega)$}
  \psfrag{PTX}[][]{\tiny$\varphi(t_1,x)$}
  \includegraphics[width=.5\linewidth]{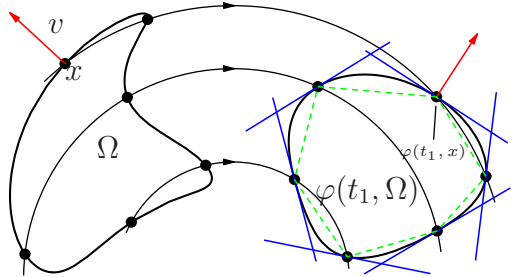}
\end{center}
\caption{\label{fig:1}Outer ({---}) and inner ({- -}) polyhedral
  approximation of convex reachable set.}
\end{figure}
Similar ideas apply to systems with inputs,
e.g.
\cite{Veliov97,HwangStipanovicTomlin03,BaturinGoncharovaPereiraSouza06}.
Thus, the question arises under what conditions reachable sets are
convex.

The more general problem of whether the image of
a nonlinear map is convex appears in the context of
optimization and optimal control
\cite{Berkovitz74,MackiStrauss82,Giannessi05} and is related to
some geometric problems with a long history
\cite{Hadamard97,Toeplitz1918,Hausdorff1919}.
Recently,
Zampieri and Gorni \cite{ZampieriGorni94} have obtained a criterion
for a local homeomorphism between open subsets of real finite
dimensional spaces to be one-to-one and to have a convex image. They
have also shown the image is convex provided that a certain
matrix is  positive semi-definite everywhere and the local
homeomorphism actually is a global $C^2$-diffeomorphism.
Polyak \cite{Polyak01} has presented a sufficient condition for the
image of a ball under a local $C^{1,1}$-submersion ($C^1$ with
Lipschitz-continuous derivative) between
real Hilbert spaces to be convex, from which a duality result
and an efficient algorithm for nonconvex optimization problems
restricted to a sufficiently small ball follow. Further sufficient
conditions for the convexity of the image of convex compact subsets of
real finite dimensional spaces under homeomorphisms
and $C^{\infty}$-subimmersions, respectively, have been presented by
Bobylev, Emel'yanov and Korovin \cite{BobylevEmelyanovKorovin04} and
Vakhrameev \cite{Vakhrameev03}.

Consider now the control system
\begin{equation}
\label{e:ControlSystem}
\dot x = f(t,x,u(t)),
\end{equation}
where
$f \colon U \subseteq \mathbb{R} \times \mathbb{R}^n\times
\mathbb{R}^m \to \mathbb{R}^n$
and $u$ is from a set $\mathcal{U}$ of admissible controls, and denote
the reachable set of \ref{e:ControlSystem} from
$\Omega \subseteq \mathbb{R}^n$ at time $t_1$ by
$\mathcal{R}(t_1,\Omega)$,
\[
\mathcal{R}(t_1,\Omega)
=
\Menge{\varphi(t_1)}{ \varphi(0) \in \Omega \text{\ and
 $\varphi \colon [0,t_1] \to \mathbb{R}^n$ is a solution of
    \ref{e:ControlSystem} for some $u \in \mathcal{U}$}}.
\]
A result of Pli{\'s} implies that, under suitable assumptions, which
include convexity conditions on $\Omega$ and on images of
$f(t,x,\cdot)$, the reachable set $\mathcal{R}(t_1,\Omega)$ is convex
for sufficiently small $t_1 > 0$ \cite{Plis74}.
{\L}ojasiewicz improved upon that result by giving an explicit upper
bound on $t_1$ \cite{Lojasiewicz78}.
For prescribed $t_1 > 0$ and $\Omega$ a singleton, Polyak has shown
under different hypotheses that $\mathcal{R}(t_1,\Omega)$ is convex if
$\mathcal{U}$ is a ball of sufficiently small radius in the space of
square integrable functions $[0,t_1] \to \mathbb{R}^m$
\cite{Polyak04}.
Recently, Azhmyakov, Flockerzi and Raisch
\cite{AzhmyakovFlockerziRaisch07} have presented a related result
for a closed-loop variant of \ref{e:ControlSystem}
to which we give a counterexample in
section \ref{s:ConvexityOfReachableSets}.
Further sufficient conditions for convexity of
the reachable set of \ref{e:ControlSystem} are known for rather
special classes of right hand sides of \ref{e:ControlSystem},
e.g. \cite{Hermes67,Lojasiewicz79,Topunov04}.

When applied to the problem described at the beginning of this
section, the results from \cite{Plis74,ZampieriGorni94,Polyak01}
cited above ensure that the reachable set $\varphi(t_1,\Omega)$ is
convex if $\Omega$ is a Euclidean ball of radius $r$, and $t_1$
\cite{Plis74,Lojasiewicz78} or $r$
\cite{Lojasiewicz78,ZampieriGorni94,Polyak01} does not exceed
some bound. However, that bound could be extremely small and of no
practical value, as is the case with the reachability problem studied
by Polyak \cite[p.~262]{Polyak04}.

In this paper, we present a necessary and sufficient condition for the
convexity of the reachable set of the ordinary differential equation
(ODE)
\begin{equation}
\label{e:ODE}
\dot x = f(t,x)
\end{equation}
from a ball $\Omega$ of initial values.
(\label{differenceinproblemstatement}Note that the uncertainty comes from a set of initial
values only. In contrast to the control system \ref{e:ControlSystem}
investigated in \cite{Plis74,Lojasiewicz78,Polyak04}, there are no
inputs to \ref{e:ODE}.)
In particular, convexity is guaranteed if $\Omega$
is sufficiently small, and we provide an upper bound on the
radius of $\Omega$, which can be directly
obtained from the right hand side $f$ of \ref{e:ODE}.
We also demonstrate by means of an example that
the balls of initial states for which our results imply the convexity
of the reachable set are large enough to be used in actual
computations, such as in local programming techniques \cite{Polyak01}
and polyhedral approximation of reachable sets discussed at the
beginning of this section. Our results extend those in
\cite{i05MMAR,i05Bosen,i06Moskau}.

The remaining of this paper is structured as follows.
After having introduced basic terminology in section \ref{s:prelim},
we establish a criterion for the convexity of a sublevel set $\Omega$,
\begin{equation}
\label{e:SublevelSet}
\Omega
=
\Menge{x \in U}{g(x)\leq 0},
\end{equation}
in terms of generalized second-order directional derivatives of $g$ in
section \ref{s:ConvexityUnderDiffeomorphisms}, where
$g \colon U \subseteq X \to \mathbb{R}$ is of class $C^{1,1}$ and
$X$ is a real Banach space.
We also present a criterion, rather than a sufficient condition, for
the image $F(\Omega)$ of $\Omega$ under a $C^{1,1}$-diffeomorphism $F$
to be convex.
In section \ref{s:ConvexityOfReachableSets} we investigate reachable
sets from a ball $\Omega$ of initial states through solutions
of the ordinary differential equation (ODE) \ref{e:ODE},
where $f\colon U \subseteq \mathbb{R} \times X \to X$
is continuous and $X$ is a real Hilbert space.
We establish a sharp upper bound on the radius of $\Omega$ that
ensures convexity of the reachable set under the assumption that $f$
is of class $C^{1,1}$ with respect to its second argument and also
present a necessary and sufficient condition for convexity
under the assumption that $f$ is of class $C^2$ with respect to its
second argument.
In section \ref{s:Application} we apply
our results to the equations of the damped mathematical pendulum.

The reader will notice that stronger smoothness assumptions than those
adopted in this paper would have simplified both notation and
arguments considerably. However, such simplification would have come
at the expense of narrowing applicability of our results since
many commonly used models of physical systems involve
$C^{1,1}$-functions that are not of class $C^2$,
e.g. \cite[Sec. 9.1]{MullerKamins86}.
On the other hand, if we had weakened smoothness requirements further,
beyond $C^{1,1}$, certain geometric properties of the boundary of
$\Omega$ and $F(\Omega)$, respectively, that are related to curvature,
would have become lost. See also the short discussion at the end of
section \ref{s:ConvexityUnderDiffeomorphisms}.
To conclude, we believe that for the problems investigated in this
paper, $C^{1,1}$-smoothness of both maps and sets is a rather natural
assumption.

\section{Preliminaries}\label{s:prelim}
Throughout this paper, ``iff'' abbreviates ``if and only if'', and
$X$ and $Y$ denote real Banach spaces with norm $\|\cdot\|$
unless specified otherwise. $\oBall(x,r)$ and $\cBall(x,r)$ denote the
open and closed, respectively, ball of radius $r$ centered at $x$,
and the space of continuous linear operators
$X \to Y$ is denoted by $\mathcal{L}(X,Y)$.

$\mathbb{R}$ and $\mathbb{R}_+$ denote the field of real numbers and
its subset of nonnegative real numbers, respectively, and
$\intcc{a,b}$, $\intoo{a,b}$,
$\intco{a,b}$, and $\intoc{a,b}$
denote the closed, open and halfopen, respectively,
intervals with end points $a$ and $b$, $a<b$.
$\sign$ denotes the signum function.
We write $y \geq x$ and $ x \leq y$ for $x,y \in \mathbb{R}^n$ if
$y-x \in \mathbb{R}^n_+$.

The domain of a map $f$ is denoted by $\dom f$,
$f \circ g$ denotes the composition of $f$ and $g$,
$(f \circ g)(x) = f(g(x))$,
$\id$ denotes the identity map, $f^{-1}$ is
used for the inverse of $f$ as well as for preimages, and $\ker f$
denotes the nullspace of $f$ if $f$ is linear.
If $L$ is $k$-linear, we set $Lh^k \defas L(h,\dots,h)$.

Arithmetic operations involving subsets of a linear space are
defined pointwise, e.g.
$\alpha M \defas \Menge{\alpha y}{y \in M}$,
$M + N \defas \Menge{y+z}{y \in M, z \in N}$ if
$\alpha \in \mathbb{R}$ and $M,N \subseteq X$.
$\boundary M$ denotes the boundary of $M \subseteq X$, and $\dim M$
denotes the dimension of a linear subspace $M \subseteq X$.

$D^jf$ denotes the derivative of order $j$ of $f$, and $D_i^j f$, the
partial derivative of order $j$ with respect to the $i$th argument of
$f$, and $D_i f \defas D_i^1 f$, $f'\defas Df \defas D^1 f$, and
$f'' \defas D^2f$.
$C^k$ denotes the class of $k$ times continuously differentiable maps,
and $C^{k,1}$, the class of maps in $C^k$ with (locally) Lipschitz-continuous
$k$th derivative.
Let $U \subseteq X$ be open. $f \colon U \to \mathbb{R}$ is a
\begriff{submersion at $x \in U$} if $f$ is of class $C^1$ on a
neighborhood of $x$ and $f'(x)$ is surjective.
$f$ is a \begriff{submersion on $V \subseteq U$} if $f$ is a
submersion at each point $x \in V$. A \begriff{$K$-submersion} is a
submersion of class $K$ whenever $K$ is one of the classes of maps
defined above.

For $f \colon U \subseteq X \to \mathbb{R}$ of class
$C^{1}$ with $U$ open, we define the
four \begriff{generalized second-order directional derivatives}
$\overline{D}^2_{\pm} f$ and
$\underline{D}^2_{\pm} f$,
\begin{align}
\label{e:upperleftrightD2}
\underline{D}^2_{\pm}f(x,h,k)
& =
\liminf_{\pm t\downarrow 0}\frac{f'(x+th)k - f'(x)k}{t},\\
\label{e:lowerleftrightD2}
\overline{D}^2_{\pm}f(x,h,k)
& =
\limsup_{\pm t\downarrow 0}\frac{f'(x+th)k - f'(x)k}{t},
\end{align}
for all $x \in U$ and all $h,k \in X$.
If $\mathcal{D}^2$ is any of the operators defined above, we note that
$\mathcal{D}^2 f$ is positively homogeneous in both its second and its
third argument and define
$\mathcal{D}^2 f(x,h^2) \defas \mathcal{D}^2 f(x,h,h)$.
Furthermore, it is easily verified that
\begin{align*}
\underline{D}_+^2 f(x,-h,-k)
&=
\underline{D}_-^2 f(x,h,k),\\
\overline{D}_+^2 f(x,-h,-k)
&=
\overline{D}_-^2 f(x,h,k).
\end{align*}

\begin{proposition}
\label{prop:ConvexityCriterionC1}
Let $U \subseteq X$ be open and convex,
$f \colon U \to \mathbb{R}$ be of class $C^1$,
and let $\mathcal{D}^2$ be one of the four operators
defined in \ref{e:upperleftrightD2}-\ref{e:lowerleftrightD2}.
Then $f$ is convex iff $\mathcal{D}^2 f(x,h^2) \geq 0$ for all $x \in U$
and all $h \in X$.
\end{proposition}
\noindent
For these and related concepts and results, see
\cite{ClarkeLedyaevSternWolenski98,Pastor05} and the references
given there.

Let $I \subseteq \mathbb{R}$ be an interval,
$U \subseteq I \times X$ be relatively open in $I \times X$, the map
$f \colon U \to X$ be (locally) Lipschitz-continuous
with respect to its second argument and continuous, and
$V
\subseteq
\Menge{(\tau,t,x) \in \mathbb{R} \times \mathbb{R}\times X}{(t,x)\in U}$.
$\varphi \colon V \to X$ is called the
\begriff{general solution} of \ref{e:ODE}
if for all $(t_0,x_0)\in U$, $\varphi(\cdot,t_0,x_0)$ is the maximal
solution of the initial value problem composed of \ref{e:ODE} and the
initial condition $x(t_0)=x_0$
\cite{Hartman02}. The map $(t,x) \mapsto \varphi(t,0,x)$ is called
the \begriff{flow} of \ref{e:ODE} if $\varphi$ is the general solution
of \ref{e:ODE} and \ref{e:ODE} is autonomous.

Let now $X$ be a Hilbert space with inner product
 \innerProd{\cdot}{\cdot}.
Two vectors $x$ and $y$ are perpendicular, $x\perp y$, if
$\innerProd{x}{y} = 0$. $\|x\|$ and $L^\ast$ denote the norm of $x$
and the adjoint of the linear map $L$, respectively, with respect to
\innerProd{\cdot}{\cdot}. We define continuous maps
$\mu_{\pm} \colon \mathcal{L}(X,X) \to \mathbb{R}$ by
\begin{align*}
\mu_-(A)
&=
\inf \Menge{\innerProd{Ax}{x}}{\|x\| = 1},\\
\mu_+(A)
&=
\sup \Menge{\innerProd{Ax}{x}}{\|x\| = 1}.
\end{align*}
The following result is
sometimes referred to as \begriff{Wa{\.z}ewski's inequality}. Its proof
given in \cite{SansoneConti64} carries over to the Hilbert space
setting.
\begin{proposition}
\label{lem:WazewskisInequality}
Let $I \subseteq \mathbb{R}$ be an interval,
$A \colon I \to \mathcal{L}(X,X)$ be continuous, and
$x \colon I \to X$ be a solution of $\dot x = A(t) x$. Then
\[
\| x(t_0) \|
\e^{\int_{t_0}^t \mu_-(A(\tau)) d\tau}
\leq
\| x(t) \|
\leq
\| x(t_0) \|
\e^{\int_{t_0}^t \mu_+(A(\tau)) d\tau}
\]
for all $t,t_0 \in I$ with $t \geq t_0$.
\end{proposition}

\section{Convexity of images of sublevel sets under diffeomorphisms}
\label{s:ConvexityUnderDiffeomorphisms}
In this section, we present a necessary and sufficient condition for
the convexity of the image $F(\Omega)$ of a sublevel set $\Omega$ from
\ref{e:SublevelSet} under a diffeomorphism $F$.
We assume that both the map $g$ from \ref{e:SublevelSet} and the
diffeomorphism $F$ are of class $C^{1,1}$; see our remarks at the end
of section \ref{s:intro}. We therefore use the generalized
derivatives defined in
\ref{e:upperleftrightD2}-\ref{e:lowerleftrightD2}; the
corresponding differentiation operators are denoted $\mathcal{D}$
throughout this section.

We first present a characterization of the convexity of sublevel sets
on which our subsequent results are based. Note that the requirement
that $\Omega$ be closed is automatically met if $U = X$.

\begin{theorem}
\label{th:ConvexityOfC11SublevelSet}
Let $U \subseteq X$ be open, $g\colon U \to\mathbb{R}$ be continuous and a
$C^{1,1}$-submersion on its zero set,
and let $\Omega$ defined by \ref{e:SublevelSet} be closed and
connected. Let further $\mathcal{D}^2$ be one of the four operators
defined in \ref{e:upperleftrightD2}-\ref{e:lowerleftrightD2}.\\
Then $\Omega$ is convex iff $\mathcal{D}^2 g(x,h^2)\geq 0$ for all
$x \in \boundary \Omega$ and all $h \in \ker g'(x)$.
\end{theorem}

\begin{proof}
Assume $\boundary \Omega \not= \emptyset$ without loss and observe that
our hypotheses
imply $g^{-1}(0) = \boundary \Omega$.
Let us call
$\mu \colon W \to \mathbb{R}$ a
{representation} of $\boundary \Omega$ about $x$
in $Z$ with respect to the direction $v$ if
$x \in \boundary \Omega$,
$v \in X$,
$W \subseteq \ker g'(x)$ is a convex open neighborhood of the origin,
and there is an open interval $V$ containing the origin such that
$Z = x + W + V v$ and
$\mu(h) \in V$ for all $h \in W$, and
\begin{equation}
\label{e:th:ConvexityOfC11SublevelSet:epi}
Z \cap \Omega
=
\Menge{x + h + \lambda v}{h \in W, \lambda \in V, \lambda \geq \mu(h)}.
\end{equation}
An application of the implicit function theorem to the
equation $g( x + h + \lambda v ) = 0$ for $h \in \ker g'(x)$ and
$\lambda \in \mathbb{R}$ shows that for all $x \in \boundary \Omega$
and all $v \in X$ with $g'(x)v < 0$ there is a representation
$\mu$ of $\boundary \Omega$ about $x$ with respect to $v$ that is of class $C^1$
and fulfills
\begin{equation}
\label{e:th:ConvexityOfC11SublevelSet:mu'}
\mu'(h) \xi
=
-\left( g'( p(h) )v\right)^{-1}
g'( p(h) ) \xi
\end{equation}
for all $h \in \dom \mu$ and all $\xi \in \ker g'(x)$,
where $\dom \mu$ denotes the domain of $\mu$ and
\[
p(h) = x + h + v \mu(h).
\]
Let $\mu$ be such a representation.
It follows from \ref{e:th:ConvexityOfC11SublevelSet:epi}
that $Z \cap \Omega$ is convex iff $\mu$ is;
see \cite{Giles82} for a proof of an analogous result on epigraphs.
Further, \ref{e:th:ConvexityOfC11SublevelSet:epi} yields $g'(x)v < 0$,
and \ref{e:th:ConvexityOfC11SublevelSet:mu'} implies that $\mu$ is
actually of class $C^{1,1}$.
We prove
\begin{equation}
\label{e:th:ConvexityOfC11SublevelSet:mu''}
\mathcal{D}^2 \mu(h,\xi^2)
=
-\left( g'( p(h) )v\right)^{-1}
\mathcal{D}^2 g\left( p(h), (p'(h)\xi)^2\right)
\end{equation}
for all $h \in \dom \mu$ and all
$\xi \in \ker g'(x)$, which is the key relation.\\
Let $\mathcal{D}^2 = \underline{D}^2_+$.
For $h = 0$, \ref{e:th:ConvexityOfC11SublevelSet:mu''} reduces to
\begin{equation}
\label{e:th:ConvexityOfC11SublevelSet:mu''reduced}
\liminf_{t\downarrow 0}\frac{g'( x + t \xi + v \mu(t \xi) ) \xi}{t}
=
\liminf_{t\downarrow 0}\frac{g'( x + t \xi ) \xi}{t}
\end{equation}
since $\mu$ and $g'$ are continuous, $g'(x)v < 0$, and $g'(x)\xi = 0$.
As $g'$ is uniformly Lipschitz-continuous in a neighborhood of $x$ and
$\mu'(0)=0$ we obtain
\[
\lim_{t \to 0}
\frac{g'( x + t \xi + v \mu(t \xi) ) \xi - g'( x + t \xi ) \xi}{t}
= 0,
\]
which implies
\ref{e:th:ConvexityOfC11SublevelSet:mu''reduced}.
Therefore, for $\mathcal{D} = \underline{D}^2_+$,
\begin{equation}
\label{e:th:ConvexityOfC11SublevelSet:mu''(0)}
\mathcal{D}^2 \kappa(0,\zeta^2)
=
-\left( g'( y )v\right)^{-1}
\mathcal{D}^2 g\left( y, \zeta^2\right)
\end{equation}
for all representations $\kappa$ of $\boundary \Omega$ about $y$
with respect to the direction $v$ and all $\zeta \in \ker g'(y)$.
For the other three operators defined in 
\ref{e:upperleftrightD2}-\ref{e:lowerleftrightD2},
\ref{e:th:ConvexityOfC11SublevelSet:mu''(0)} is obtained in
exactly the same way.\\
Let now $h \in \dom \mu$ be arbitrary, let $P$ be the projection
operator along $v$ onto $\ker g'(x)$, let $y = x + h + v \mu(h)$,
and define $\kappa$ on a neighborhood of the origin in $\ker g'(y)$ by
$
v \kappa(s)
=
v \mu( h + P s ) - v \mu(h) - (\id - P)s
$,
which implies
\begin{equation}
\label{e:th:ConvexityOfC11SublevelSet:kappa''}
\mathcal{D}^2 \kappa(0,\zeta^2)
=
\mathcal{D}^2 \mu(h,(P \zeta)^2)
\end{equation}
for all $\zeta \in \ker g'(y)$.
Choose convex open neighborhoods of the origin
$W' \subseteq \ker g'(y)$ and $V' \subseteq \mathbb{R}$ such that
$Z' \defas y + W' + V' v \subseteq Z$ and $\kappa(s) \in V'$ whenever
$s \in W'$. It is easily verified that the restriction of $\kappa$ to
$Z'$ is a representation of $\boundary \Omega$ about $y$ in $Z'$
with respect to $v$.
Hence, \ref{e:th:ConvexityOfC11SublevelSet:kappa''} and
\ref{e:th:ConvexityOfC11SublevelSet:mu''(0)} for
$\zeta = p'(h)\xi$ give \ref{e:th:ConvexityOfC11SublevelSet:mu''}.\\
With \ref{e:th:ConvexityOfC11SublevelSet:mu''} at our disposal, we are
now in a position to prove the theorem.
It follows from \ref{e:th:ConvexityOfC11SublevelSet:mu''} that
$\mathcal{D}^2 g(x,h^2)\geq 0$ for all $x \in \boundary \Omega$ and
all $h \in \ker g'(x)$ iff for all $x \in \boundary \Omega$ there is a
representation $\mu$ of $\boundary \Omega$ about $x$ with
$\mathcal{D}^2 \mu(h,\xi^2)\geq 0$ for all $h \in \dom \mu$ and
all $\xi \in \ker g'(x)$. By Prop.~\ref{prop:ConvexityCriterionC1},
the latter condition is equivalent to the
convexity of $\mu$, which in turn is equivalent to the
convexity of $Z \cap \Omega$ for some neighborhood $Z \subseteq X$ of
$x$. As $\Omega$ is closed and connected, application of a
(generalization of a) theorem of Tietze-Nakajima \cite{Valentine64}
completes the proof.
\end{proof}

The criterion for the convexity of the set $\Omega$ presented in
Theorem \ref{th:ConvexityOfC11SublevelSet} takes the form of a
condition on the map $g$ that defines $\Omega$. In finite dimensions,
i.e., if $X = \mathbb{R}^n$, the oriented distance function of
$\Omega$ is a $C^{1,1}$-submersion on its zero set under our assumptions
\cite{DelfourZolesio04}, and hence, is a natural choice for the map
$g$ in Theorem \ref{th:ConvexityOfC11SublevelSet}.
However, the condition presented in Theorem
\ref{th:ConvexityOfC11SublevelSet} actually describes metric
properties of the boundary of $\Omega$ rather than properties of
maps defining $\Omega$. In particular, if
$X = \mathbb{R}^n$ and the map $g$ is of
class $C^2$, then $\mathcal{D}^2 g(x,h^2) = g''(x)h^2$, and the
restriction of $g''(x)$ to $\ker g'(x)$ coincides with the second
fundamental form \cite{Thorpe79} of $\boundary \Omega$ at $x$ up to a
positive scalar factor. Hence, Theorem
\ref{th:ConvexityOfC11SublevelSet} implies the following well-known
result: A closed, connected set of class $C^2$
is convex iff the second fundamental
form of its boundary is positive semi-definite everywhere.

Next we present a criterion for the convexity of the image of a
sublevel set $\Omega$ under a diffeomorphism. Note that the
requirement that $F(\Omega)$ be closed is automatically met if
$\Omega$ is compact or $V = Y$. We do not assume that $\Omega$ itself
is convex.

\begin{theorem}
\label{th:ConvexityCriterionC11}
Let $U \subseteq X$ and $V \subseteq Y$ be open,
$g\colon U \to\mathbb{R}$ be continuous and a $C^{1,1}$-submersion on
its zero set,
let $\Omega$ defined by \ref{e:SublevelSet} be closed and
connected,
$F\colon U\to V$ be a $C^{1,1}$-diffeomorphism, and $F(\Omega)$ be
closed.
Let further $\mathcal{D}^2$ be one of the four operators
defined in \ref{e:upperleftrightD2}-\ref{e:lowerleftrightD2}.\\
Then $F(\Omega)$ is convex iff
\begin{equation}
\label{e:th:ConvexityCriterionC11:1}
\mathcal{D}^2\left( g'(x) F'(x)^{-1} F(\cdot) \right)(x,h^2)
\leq
\mathcal{D}^2 g(x,h^2)
\end{equation}
for all $x \in \boundary \Omega$ and all $h \in \ker g'(x)$.
\end{theorem}

\begin{proof}
Under our hypotheses, $F(\Omega)$ is closed and connected,
$F(\boundary \Omega) = \boundary F(\Omega)$, and
\[
F(\Omega)
=
\Menge{y \in V}{f(y)\leq 0}
\]
for $f \defas g \circ F^{-1} \colon V \to \mathbb{R}$. In addition,
$F'(x)h \in \ker f'(F(x))$ iff $h \in \ker g'(x)$, for all
$x \in \boundary \Omega$. Therefore, by
Theorem \ref{th:ConvexityOfC11SublevelSet}, $F(\Omega)$ is convex iff
\begin{equation}
\label{e:th:ConvexityCriterionC11:0}
d^2 f(F(x),(F'(x)h)^2) \geq 0
\end{equation}
for all $x \in \boundary \Omega$ and all $h \in \ker g'(x)$, whenever
$d^2$ is one of the four operators defined in
\ref{e:upperleftrightD2}-\ref{e:lowerleftrightD2}.
We first establish the relations
\begin{align}
\label{e:th:ConvexityCriterionC11:A}
\underline{D}_{\pm}^2 f(F(x),(F'(x)h)^2)
&\leq
\underline{D}_{\pm}^2 g(x,h^2)
-
\underline{D}_{\pm}^2 (g'(x)F'(x)^{-1}F(\cdot))(x,h^2),\\
\label{e:th:ConvexityCriterionC11:B}
\overline{D}_{\pm}^2 f(F(x),(F'(x)h)^2)
& \geq
\underline{D}_{\pm}^2 g(x,h^2)
-
\underline{D}_{\pm}^2 (g'(x)F'(x)^{-1}F(\cdot))(x,h^2),\\
\label{e:th:ConvexityCriterionC11:C}
\underline{D}_{\pm}^2 f(F(x),(F'(x)h)^2)
& \leq
\overline{D}_{\pm}^2 g(x,h^2)
-
\overline{D}_{\pm}^2 (g'(x)F'(x)^{-1}F(\cdot))(x,h^2),\\
\label{e:th:ConvexityCriterionC11:D}
\overline{D}_{\pm}^2 f(F(x),(F'(x)h)^2)
& \geq
\overline{D}_{\pm}^2 g(x,h^2)
-
\overline{D}_{\pm}^2 (g'(x)F'(x)^{-1}F(\cdot))(x,h^2).
\end{align}
Let $x \in \boundary \Omega$ and $h \in \ker g'(x)$. We assume $X = Y$,
$x = F(x) = 0$, and $F'(0) = \id$ without loss of generality to obtain
\begin{align}
\label{e:th:ConvexityCriterionC11:p1}
\underline{D}_{\pm}^2 f(0,h^2)
&=
\liminf_{\pm t \downarrow 0}\frac{f'(th)h}{t},\\
\label{e:th:ConvexityCriterionC11:p2}
\overline{D}_{\pm}^2 f(0,h^2)
&=
\limsup_{\pm t \downarrow 0}\frac{f'(th)h}{t},\\
\label{e:th:ConvexityCriterionC11:pNach2}
\underline{D}_{\pm}^2 (g'(0)F(\cdot))(0,h^2)
&=
\liminf_{\pm t \downarrow 0}\frac{f'(0)F'(th)h}{t}.
\end{align}
Continuity of $f'$ and Lipschitz-continuity of $F'$ imply
$
\lim\limits_{t \to 0}
%\frac{(f'(th) - f'(0))(F'(th)h - h)}{t}
(f'(th) - f'(0))(F'(th)h - h)/t
=0
$,
hence
\begin{equation}
\label{e:th:ConvexityCriterionC11:p3}
\underline{D}_{\pm}^2 g(0,h^2)
=
\liminf_{\pm t \downarrow 0}
\frac{f'(F(th))F'(th)h}{t}
=
\liminf_{\pm t \downarrow 0}
\left(
\frac{f'(th)h}{t}
+
\frac{f'(0)F'(th)h}{t}
\right).
\end{equation}
\ref{e:th:ConvexityCriterionC11:A} and
\ref{e:th:ConvexityCriterionC11:B} follow from
\ref{e:th:ConvexityCriterionC11:p1},
\ref{e:th:ConvexityCriterionC11:p2},
\ref{e:th:ConvexityCriterionC11:pNach2}, and
\ref{e:th:ConvexityCriterionC11:p3}.
\ref{e:th:ConvexityCriterionC11:C} and
\ref{e:th:ConvexityCriterionC11:D} are shown by analogous arguments.\\
Assume now $F(\Omega)$ is convex. Then
\ref{e:th:ConvexityCriterionC11:0}, \ref{e:th:ConvexityCriterionC11:A} and
\ref{e:th:ConvexityCriterionC11:C} imply
\ref{e:th:ConvexityCriterionC11:1}. Conversely,
\ref{e:th:ConvexityCriterionC11:1}, \ref{e:th:ConvexityCriterionC11:B} and
\ref{e:th:ConvexityCriterionC11:D} imply
\ref{e:th:ConvexityCriterionC11:0} for at least one of the operators
defined in \ref{e:upperleftrightD2}-\ref{e:lowerleftrightD2}, and
hence, $F(\Omega)$ is convex.
\end{proof}

\begin{corollary}
\label{cor:ConvexityCriterionC11:Ellipsoid}
Let $X$, $Y$ be real Hilbert spaces,
$U \subseteq X$ and $V \subseteq Y$ be open,
$F \colon U \to V$ be a $C^{1,1}$-diffeomorphism,
$\Omega \subseteq U$ be a closed ball centered at $x_0$, $F(\Omega)$ be
closed, and $\mathcal{D}^2$ be one of the four operators defined in
\ref{e:upperleftrightD2}-\ref{e:lowerleftrightD2}.\\
Then $F(\Omega)$ is convex iff
\begin{equation}
\label{e:cor:ConvexityCriterionC11:Ellipsoid:1}
\mathcal{D}^2
\innerProd{x-x_0}{F'(x)^{-1}F(\cdot)}(x,h^2)
\leq 1
\end{equation}
for all $x \in \boundary \Omega$ and all $h \perp (x-x_0)$ with
$\|h\|=1$.
\end{corollary}
\begin{proof}
Set $g(x) \defas \|x - x_0\|^2 - r^2$, where $r$ is the radius of
$\Omega$, and apply Theorem \ref{th:ConvexityCriterionC11}.
\end{proof}

In contrast to related results in
\cite{ZampieriGorni94,Polyak01a,BobylevEmelyanovKorovin04},
the conditions in Theorem \ref{th:ConvexityCriterionC11} and in
Corollary \ref{cor:ConvexityCriterionC11:Ellipsoid}
are to be
checked on the boundary of $\Omega$ and for tangent vectors
only.
Further, as with Theorem \ref{th:ConvexityOfC11SublevelSet},
the criteria in Theorem \ref{th:ConvexityCriterionC11}
and Corollary \ref{cor:ConvexityCriterionC11:Ellipsoid} take
particularly simple forms if the maps $F$ and $g$ are smooth. If $F$
is of class $C^2$, the left hand side of
\ref{e:th:ConvexityCriterionC11:1} and
\ref{e:cor:ConvexityCriterionC11:Ellipsoid:1} equals
$g'(x) F'(x)^{-1} F''(x)h^2$ and
\innerProd{x-x_0}{F'(x)^{-1} F''(x)h^2},
respectively, and if $g$ is of class $C^2$, then
$\mathcal{D}^2 g(x,h^2) = g''(x)h^2$ in Theorem
\ref{th:ConvexityCriterionC11}.

The following is an immediate consequence of
Corollary \ref{cor:ConvexityCriterionC11:Ellipsoid}:
The image of a ball centered at the origin under a
$C^{1,1}$-diffeomorphism defined on a neighborhood of the origin in a
Hilbert space is convex provided the radius of the ball is
sufficiently small \cite{ZampieriGorni94,Polyak01a}. The assumption
that the diffeomorphism be of class $C^{1,1}$ rather than merely $C^1$
is essential \cite{ZampieriGorni94}. For $C^2$-diffeomorphisms
the result is obvious \cite{ZampieriGorni94} and can also be concluded
from a well-known result on the existence of geodetically convex
neighborhoods in Riemannian manifolds; see \cite{KobayashiNomizu63}
and also \cite[Sec.~6]{ZampieriGorni94}.
Polyak has raised the question of whether the result extends
to uniformly convex Banach spaces \cite{Polyak03}. It does not,
as the following example shows.
\begin{example}
\label{ex:CounterExUniformlyConvexBSs}
Endow $\mathbb{R}^2$ with the norm $\| \cdot \|$ defined by
$
\| x \| = \left( |x_1|^p + |x_2|^p \right)^{1/p}
$
for some real $p > 2$, which makes $\mathbb{R}^2$ a uniformly convex
space. Then
$
\cBall(0,r)
=
\Menge{x \in \mathbb{R}^2}{g(x) \leq 0}
$ for arbitrary $r > 0$,
where we have set $g(x) = \| x \|^p - r^p$.
The map $g$ is of class $C^2$.
For $x = (r,0) \in \boundary \oBall(0,r)$,
$h = (0,1) \in \ker g'(x)$,
$\varepsilon > 0$, and
$F \colon x \mapsto(x_1 + \varepsilon x_2^2, x_2)$
we obtain
$g''(x)h^2 = 0$ and
$g'(x) F'(x)^{-1} F''(x)h^2 = 2 \varepsilon p r^{p-1}$.
By Theorem \ref{th:ConvexityCriterionC11}, $F( \cBall(0,r) )$ is not
convex, no matter how small $r$ and $\varepsilon$ are.
\end{example}

We finally remark that the assumption in Theorem
\ref{th:ConvexityCriterionC11} and Corollary
\ref{cor:ConvexityCriterionC11:Ellipsoid} that $F$ be a global
diffeomorphism on a neighborhood of $\Omega$ could easily be relaxed.

\section{Convexity of reachable sets}
\label{s:ConvexityOfReachableSets}
In this section, we investigate reachable
sets from a set $\Omega$ of initial states through solutions
of the ordinary differential equation (ODE) \ref{e:ODE},
where the right hand side
$f\colon U \subseteq \mathbb{R} \times X \to X$ of \ref{e:ODE} is
of class $C^{1,1}$ or $C^2$ with respect to its second argument and
continuous.
We also give a counterexample to a related result for control systems
recently presented by Azhmyakov, Flockerzi and Raisch
\cite{AzhmyakovFlockerziRaisch07}.
For the sake of simplicity, we restrict ourselves to the case where
$\Omega$ is a ball and $X$ is a real Hilbert space, so that Corollary
\ref{cor:ConvexityCriterionC11:Ellipsoid} applies.

The general solution of \ref{e:ODE} is denoted $\varphi$ throughout
this section.
In order to avoid confusion of ideas, we would like to remind the
reader that according to the notation adopted in section
\ref{s:prelim}, operators $D_2$ and $D_3$ refer to the partial
derivative of a map with respect to its second and third,
respectively, argument, whereas $D^2$ refers to the second order
derivative, and $D_2^2$ and $D_3^2$, to the second order partial
derivative with respect to the second and third, respectively,
argument.

Our first result is a sufficient condition for the set of states
reachable from a ball of initial values to be convex.
That condition takes the form of an upper bound on the radius of the
ball of initial values, which will be shown to be sharp later in
Corollary
\ref{cor:th:SufficientConvexityConditionC11NormODE:Ellipsoid:sharpness}.

\begin{theorem}
\label{th:SufficientConvexityConditionC11NormODE:Ellipsoid}
Let $I \subseteq \mathbb{R}$ be an interval,
$U \subseteq I \times X$ be relatively open in $I \times X$, and the
right hand side $f: U \to X$ of \ref{e:ODE} be
of class $C^{1,1}$ with respect to its second argument and continuous.
Let further $x_0 \in X$, $r > 0$ and $t_0,t_1 \in I$ be such
that $\{t_1\} \times \{t_0\} \times \cBall(x_0,r) \subseteq \dom \varphi$.
Finally, assume there are $M_1, M_2 \in \mathbb{R}$ that
\begin{align}
\label{e:th:SufficientConvexityConditionC11NormODE:Ellipsoid:M1}
M_1
& \geq
\begin{cases}
2 \mu_+ \left(D_2f(\tau,x)\right)
-
\mu_-
\left(D_2f(\tau,x)\right)
, & \text{if $t_1 \geq t_0$},\\
\mu_+ \left(D_2f(\tau,x)\right)
-
2 \mu_-
\left(D_2f(\tau,x)\right)
, & \text{otherwise},
\end{cases}\\
\label{e:th:SufficientConvexityConditionC11NormODE:Ellipsoid:M2}
M_2
& \geq
\limsup_{h\to 0}
\frac{\|D_2f(\tau,x + h) - D_2f(\tau,x)\|}{\|h\|}\\
\intertext{holds for all $(\tau,x) \in U$, and define $K$ by}
\label{e:th:SufficientConvexityConditionC11NormODE:Ellipsoid:K}
K(\alpha)
& =
\begin{cases}
|t_1 - t_0|, & \text{if $\alpha = 0$},\\
\left(\exp(\alpha |t_1 - t_0|) - 1\right)/\alpha, & \text{otherwise}.
\end{cases}
\end{align}
Then the reachable set $\varphi(t_1,t_0,\cBall(x_0,r))$ is convex if
\begin{equation}
\label{th:SufficientConvexityConditionC11NormODE:Ellipsoid:TheCondition}
r M_2 K(M_1) \leq 1.
\end{equation}
\end{theorem}

\begin{proof}
Pick arbitrary $x \in \boundary \oBall(x_0,r)$ and $h \perp (x - x_0)$
with $\|h\|=1$. We show that
\ref{e:cor:ConvexityCriterionC11:Ellipsoid:1} holds for
$F \defas \varphi(t_1,t_0,\cdot)$, which implies
$\varphi(t_1,t_0,\cBall(x_0,r))$
is convex by Corollary \ref{cor:ConvexityCriterionC11:Ellipsoid}.
To this end, we assume without loss of generality
$t_0 = 0$, $t_1 \not= 0$,
$x = 0$, and $I = \intcc{0,t_1}$ if $t_1 > 0$ and
$I = \intcc{t_1,0}$, otherwise.

As $I$ is compact, $\varphi$ is continuous, and $U$ is relatively open
in $I \times X$, we may choose $\delta > 0$ such that
$(\tau,\varphi(\tau,0,0) + h) \in U$ for all $\tau \in I$ and all
$h \in \oBall(0,\delta)$.
Then \ref{e:th:SufficientConvexityConditionC11NormODE:Ellipsoid:M2}
implies
\begin{equation}
\label{e:th:SufficientConvexityConditionC11NormODE:Ellipsoid:p:1a}
\| D_2f(\tau, \varphi(\tau,0,0) + h) - D_2f(\tau, \varphi(\tau,0,0)) \|
\leq
M_2 \|h\|
\end{equation}
for all $\tau \in I$ and all $h \in \oBall(0,\delta)$.

Pick arbitrary $\varepsilon > 0$.
By the relative openness of the domain of $\varphi$ in
$I \times I \times X$, the continuity of $\varphi$ and
$D_3 \varphi$, and the compactness of $I$, there is some neighborhood
$W \subseteq \oBall(0,\delta)$ of the origin in $X$
such that for all $w \in W$ and all $\tau \in I$ we have
$I \subseteq \dom \varphi(\cdot,0,w)$ as well as the following
estimates:
\begin{align}
\label{e:th:SufficientConvexityConditionC11NormODE:Ellipsoid:p:2}
\| D_3 \varphi(\tau,0,w) \|
& \leq
(1+\varepsilon) \| D_3 \varphi(\tau,0,0) \|,\\
\label{e:th:SufficientConvexityConditionC11NormODE:Ellipsoid:p:5}
\| \varphi(\tau,0,w) - \varphi(\tau,0,0) \|
& <
\delta.
\end{align}
Define maps $Z$ and $A$ on $I \times W$ by
\begin{align*}
Z(\tau,w)
&=
D_3 \varphi(\tau,0,w) - D_3 \varphi(\tau,0,0),\\
A(\tau,w)
&=
\left(
D_2f(\tau,\varphi(\tau,0,w)) - D_2f(\tau,\varphi(\tau,0,0))
\right)
D_3 \varphi(\tau,0,w)
\end{align*}
and use the variational equation of \ref{e:ODE} along
$\varphi(\cdot,0,w)$ and $\varphi(\cdot,0,0)$, respectively, to obtain
\begin{equation}
\label{e:th:SufficientConvexityConditionC11NormODE:Ellipsoid:p:LinDGL}
D_1 Z(\tau,w)
=
D_2f(\tau,\varphi(\tau,0,0)) Z(\tau,w) + A(\tau,w)
\end{equation}
for all $(\tau,w)\in I \times W$.
\ref{e:th:SufficientConvexityConditionC11NormODE:Ellipsoid:p:LinDGL}
is a linear differential equation in $Z(\cdot,w)$, and
$Z(0,\cdot) = 0$. Hence
$
Z(t_1,w)
=
D_3 \varphi(t_1,0,0)
\int_0^{t_1}
D_3 \varphi(\tau,0,0)^{-1} A(\tau,w) d\tau
$,
which implies
\begin{equation}
\label{e:th:SufficientConvexityConditionC11NormODE:Ellipsoid:p:6}
\left\| D_3 \varphi(t_1,0,0)^{-1} Z(t_1,w) \right\|
\leq
\left|
\int_0^{t_1}
\left\|D_3 \varphi(\tau,0,0)^{-1}\right\| \cdot \left\|A(\tau,w)\right\| d\tau
\right|
\end{equation}
for all $w \in W$.
Use the mean value theorem to estimate
$\| \varphi(\tau,0,w) - \varphi(\tau,0,0) \|$ and then apply
\ref{e:th:SufficientConvexityConditionC11NormODE:Ellipsoid:p:1a},
\ref{e:th:SufficientConvexityConditionC11NormODE:Ellipsoid:p:2}, and
\ref{e:th:SufficientConvexityConditionC11NormODE:Ellipsoid:p:5}
to obtain
\begin{equation}
\label{e:th:SufficientConvexityConditionC11NormODE:Ellipsoid:p:6b}
\|A(\tau,w)\|
\leq
(1+\varepsilon)^2 M_2 \|D_3 \varphi(\tau,0,0)\|^2 \|w\|.
\end{equation}
From the variational equation of \ref{e:ODE} along
$\varphi(\cdot,0,0)$, its adjoint,
\ref{e:th:SufficientConvexityConditionC11NormODE:Ellipsoid:M1},
Wa{\.z}ewski's inequality
(Prop.~ \ref{lem:WazewskisInequality}), and
\ref{e:th:SufficientConvexityConditionC11NormODE:Ellipsoid:K} we
obtain
$
|\int_0^{t_1}
\|D_3 \varphi(\tau,0,0)^{-1}\| \cdot \|D_3 \varphi(\tau,0,0)\|^2
d\tau|
\leq
K(M_1)$,
regardless of the sign of $t_1$, so that
$
\left\| F'(0)^{-1} ( F'(w) - F'(0) ) \right\|
\leq
(1 + \varepsilon)^2 M_2 K(M_1) \|w\|
$
for all $w \in W$ by
\ref{e:th:SufficientConvexityConditionC11NormODE:Ellipsoid:p:6} and
\ref{e:th:SufficientConvexityConditionC11NormODE:Ellipsoid:p:6b}. Now
let $\varepsilon$ tend to $0$ to obtain
\ref{e:cor:ConvexityCriterionC11:Ellipsoid:1} from $r M_2 K(M_1) \leq 1$.
\end{proof}

The following is an immediate consequence of Theorem
\ref{th:SufficientConvexityConditionC11NormODE:Ellipsoid}.

\begin{corollary}
\label{cor:th:SufficientConvexityConditionC11NormODE:Ellipsoid}
Let $U$, $f$, $x_0$, $r$, $t_0$, $t_1$, $M_2$ and $K$ as in Theorem
\ref{th:SufficientConvexityConditionC11NormODE:Ellipsoid},
assume there are $\lambda_-, \lambda_+ \in \mathbb{R}$ that
\begin{equation}
\label{e:th:SufficientConvexityConditionC11NormODE:Ellipsoid:L+-}
\lambda_-
\leq
\mu_-
\left(D_2f(\tau,x)\right)
\leq
\mu_+
\left(D_2f(\tau,x)\right)
\leq
\lambda_+
\end{equation}
holds for all $(\tau,x) \in U$, and define $M_1$ by
\begin{equation}
\label{e:th:SufficientConvexityConditionC11NormODE:Ellipsoid:M1'}
M_1
=
\begin{cases}
2 \lambda_+ - \lambda_-, & \text{if $t_1 \geq t_0$},\\
\lambda_+ - 2 \lambda_-, & \text{otherwise}.
\end{cases}
\end{equation}
Then the reachable set $\varphi(t_1,t_0,\cBall(x_0,r))$ is convex if
\ref{th:SufficientConvexityConditionC11NormODE:Ellipsoid:TheCondition}
holds.
\end{corollary}

The main advantage of Theorem
\ref{th:SufficientConvexityConditionC11NormODE:Ellipsoid} over the
results from section \ref{s:ConvexityUnderDiffeomorphisms} is that the
bound on the radius can be
determined directly from properties of the right hand side of
\ref{e:ODE}. Note that Theorem
\ref{th:SufficientConvexityConditionC11NormODE:Ellipsoid} \textit{cannot} be
obtained from applying
any of the estimates from the literature
\cite{ZampieriGorni94,Polyak01a,BobylevEmelyanovKorovin04} to the map
$F \defas \varphi(t_1,t_0,\cdot)$. In fact, separately estimating
$\|F'(x)^{-1}\|$ and $\|F''(x)\|$ gives a larger bound in general.

We would like to comment on our hypotheses.
First note that 
$\mu_-(A)$ and $\mu_+(A)$ equal the minimum and maximum,
respectively, eigenvalues of the self-adjoint part
\begin{equation}
\label{e:selfadjointpart}
\frac{1}{2}( A + A^\ast )
\end{equation}
of $A$ if $X = \mathbb{R}^n$.
Hence, \ref{e:th:SufficientConvexityConditionC11NormODE:Ellipsoid:M1}
and \ref{e:th:SufficientConvexityConditionC11NormODE:Ellipsoid:L+-}
reduce to bounds on eigenvalues of \ref{e:selfadjointpart}. If the
ball $\cBall(x_0,r)$ of initial values in Theorem
\ref{th:SufficientConvexityConditionC11NormODE:Ellipsoid} and
Corollary
\ref{cor:th:SufficientConvexityConditionC11NormODE:Ellipsoid} is an
Euclidean ball, or equivalently, if the inner product
\innerProd{\cdot}{\cdot} equals the Euclidean inner product
\innerProdEuclid{\cdot}{\cdot} given by
\begin{equation}
\label{e:StandardInnerProd}
\innerProdEuclid{x}{y} = \sum_{i=1}^n x_i y_i,
\end{equation}
then $A^\ast$ is just the transpose $A^T$ of $A$. If the ball
$\cBall(x_0,r)$ is an ellipsoid rather than Euclidean, then there
is a symmetric positive definite matrix $Q$ such that
$\innerProd{x}{y} = \innerProdEuclid{x}{Qy}$, from which
$A^\ast = Q^{-1}A^T Q$ follows.
This shows conditions
\ref{e:th:SufficientConvexityConditionC11NormODE:Ellipsoid:M1}
and \ref{e:th:SufficientConvexityConditionC11NormODE:Ellipsoid:L+-}
can be readily verified.

Second, note that in the Euclidean case, the result of {\L}ojasiewicz
\cite{Lojasiewicz78} for the control system \ref{e:ControlSystem}
applied to the special case investigated in this section would result
in a similar bound on radii. In fact, apart from some technical
hypotheses, the sufficient condition
\ref{th:SufficientConvexityConditionC11NormODE:Ellipsoid:TheCondition}
for the convexity of the reachable set is obtained from
\cite{Lojasiewicz78}, with $M_1$ defined by
$M_1 \geq 3 \|D_2f(\tau,x)\|$ rather than by
\ref{e:th:SufficientConvexityConditionC11NormODE:Ellipsoid:M1} or
\ref{e:th:SufficientConvexityConditionC11NormODE:Ellipsoid:M1'}.
This gives a bound on the radius $r$ that is never larger and in
general smaller than the bounds presented in Theorem
\ref{th:SufficientConvexityConditionC11NormODE:Ellipsoid} and
Corollary
\ref{cor:th:SufficientConvexityConditionC11NormODE:Ellipsoid} since
$2 \mu_{+}(A) - \mu_{-}(A) \leq \frac{3}{2} \|A + A^T\| \leq 3 \|A\|$.

Third, note that if $U = I \times \mathbb{R}^n$ in Theorem
\ref{th:SufficientConvexityConditionC11NormODE:Ellipsoid} or in
Corollary
\ref{cor:th:SufficientConvexityConditionC11NormODE:Ellipsoid}, then
the condition
$\{t_1\} \times \{t_0\} \times \cBall(x_0,r) \subseteq \dom \varphi$
is automatically fulfilled. Indeed, it is straightforward to obtain
$\dom \varphi = I \times I \times \mathbb{R}^n$ from condition
\ref{e:th:SufficientConvexityConditionC11NormODE:Ellipsoid:M1} in this
case.

Finally, note that condition
\ref{e:th:SufficientConvexityConditionC11NormODE:Ellipsoid:M2} is just
a bound on $\|D_2^2f(\tau,x)\|$ in case $f$ is of class $C^2$ with
respect to its second argument. For this case the proof
of Theorem \ref{th:SufficientConvexityConditionC11NormODE:Ellipsoid}
can be specialized to obtain a necessary and sufficient condition:

\begin{theorem}
\label{th:ConvexityCriterionC2ODE:Ellipsoid}
Let $U$, $f$, $x_0$, $r$, $t_0$ and $t_1$ as in Theorem
\ref{th:SufficientConvexityConditionC11NormODE:Ellipsoid} and assume
in addition that $f$ is of class $C^2$ with respect to its
second argument.
Then $\varphi(t_1,t_0,\cBall(x_0,r))$ is convex iff
\begin{equation}
\label{e:th:ConvexityCriterionC2ODE:Ellipsoid:1}
\int_{t_0}^{t_1}
\innerProd{x-x_0}{
D_3\varphi(\tau,t_0,x)^{-1}
D_2^2f(\tau,\varphi(\tau,t_0,x))(D_3\varphi(\tau,t_0,x)h)^2
}
 d\tau
\leq 1
\end{equation}
for all $x \in \boundary \oBall(x_0,r)$ and all $h \perp (x-x_0)$ with
$\|h\|=1$.
\end{theorem}

\begin{proof}
Pick arbitrary $x \in \boundary \oBall(x_0,r)$ and $h \perp (x - x_0)$
with $\|h\| = 1$. We show that the left hand sides of
\ref{e:th:ConvexityCriterionC2ODE:Ellipsoid:1} and
\ref{e:cor:ConvexityCriterionC11:Ellipsoid:1} coincide if
$F = \varphi(t_1,t_0,\cdot)$.

As $F$ is of class $C^2$, we obtain
$
F'(x)^{-1}F''(x)h^2
=
D_3\varphi(t_1,t_0,x)^{-1} D_3^2\varphi(t_1,t_0,x)h^2
$.
Since $D_3\varphi(\cdot,t_0,x)$ is a solution of the variational
equation, $D_3\varphi(\cdot,t_0,x)h$ solves the initial value problem
\begin{align*}
\dot z(\tau)
&=
D_2f(\tau,\varphi(\tau,t_0,x))z(\tau),\\
z(t_0)
&=
h,\\
\intertext{which implies that $D_3^2\varphi(\cdot,t_0,x)h^2$ solves}
\dot z(\tau)
&=
D_2f(\tau,\varphi(\tau,t_0,x))z(\tau)
+ D_2^2f(\tau,\varphi(\tau,t_0,x))\left( D_3\varphi(\tau,t_0,x)h \right)^2,\\
z(t_0)
&=
0.
\end{align*}
Hence,
$
F'(x)^{-1}F''(x)h^2
=
\int_{t_0}^{t_1}
D_3\varphi(\tau,t_0,x)^{-1}
D_2^2f(\tau,\varphi(\tau,t_0,x))\left( D_3\varphi(\tau,t_0,x)h
\right)^2
d\tau
$.
\end{proof}

In contrast to the hypotheses of Theorem
\ref{th:SufficientConvexityConditionC11NormODE:Ellipsoid} and
Corollary
\ref{cor:th:SufficientConvexityConditionC11NormODE:Ellipsoid}, which
may be verified by direct inspection of the right hand side $f$ of
\ref{e:ODE},
the necessary and sufficient condition
for convexity of the image $\varphi(t_1,t_0,\cBall(x_0,r))$ of the
diffeomorphism $\varphi(t_1,t_0,\cdot)$ that is established in Theorem
\ref{th:ConvexityCriterionC2ODE:Ellipsoid} contains the
diffeomorphism itself.
The advantage of the latter result
over a direct application of (the $C^2$-version of) Corollary
\ref{cor:ConvexityCriterionC11:Ellipsoid}
is that the second derivative of that diffeomorphism
does not appear in condition
\ref{e:th:ConvexityCriterionC2ODE:Ellipsoid:1}. Hence, in order to
estimate the left hand side of
\ref{e:th:ConvexityCriterionC2ODE:Ellipsoid:1} for a particular
example of \ref{e:ODE}, one has to study the variational
equation of \ref{e:ODE} and its adjoint only.
That way one may obtain bounds on the radius strictly
greater than the one given in Theorem
\ref{th:SufficientConvexityConditionC11NormODE:Ellipsoid}. This is
demonstrated in section \ref{s:Application}.

With the criterion from Theorem
\ref{th:ConvexityCriterionC2ODE:Ellipsoid} at our disposal, we are now
able to prove that the bound on the radius given in Corollary
\ref{cor:th:SufficientConvexityConditionC11NormODE:Ellipsoid} is
sharp, from which the sharpness of the bound in Theorem
\ref{th:SufficientConvexityConditionC11NormODE:Ellipsoid} easily
follows.

\begin{corollary}
\label{cor:th:SufficientConvexityConditionC11NormODE:Ellipsoid:sharpness}
Let there be given
$t_1,t_0,\lambda_-,\lambda_+,M_2, r \in \mathbb{R}$, where
$\lambda_- < \lambda_+$, $M_2 >0$, let $M_1$ and $K$ be
defined by
\ref{e:th:SufficientConvexityConditionC11NormODE:Ellipsoid:M1'} and
\ref{e:th:SufficientConvexityConditionC11NormODE:Ellipsoid:K}, and
assume
$\dim X \geq 2$ and $r M_2 K(M_1) > 1$.\\
Then there are an autonomous ODE \ref{e:ODE} with analytic right hand
side $f$ defined on $\mathbb{R} \times X$ such that
\ref{e:th:SufficientConvexityConditionC11NormODE:Ellipsoid:L+-} and
\ref{e:th:SufficientConvexityConditionC11NormODE:Ellipsoid:M2} hold
for all $(\tau,x) \in \mathbb{R} \times X$, yet
$\varphi(t_1,t_0,\cBall(0,r))$ is not convex.
\end{corollary}

\begin{proof}
Assume $X = \mathbb{R}^2$ and $t_0 = 0$ without loss of generality.
Since $K$ is continuous and $\lambda_- < \lambda_+$ we may pick
$\varepsilon \in \intoo{0,\lambda_+ - \lambda_-}$
such that $r M_2 K(M_1-\varepsilon) > 1$.
Set
$\nu_{\pm}
=
\lambda_{\pm} \mp \varepsilon/3$
and define $g$ and $f$ by
\begin{align*}
g(s)
&=
\alpha^2 (2 M_2)^{-1}
\left(\sin\left( M_2 s/\alpha \right)\right)^2,\\
f(\tau,x)
&=
(\nu_- x_1 + g(x_2), \nu_+ x_2)
\end{align*}
to obtain $\|D_2^2f(\tau,x)\|\leq M_2$ for all $x$ and
any $\alpha > 0$. By our choice of $\nu_{\pm}$
and the continuity of eigenvalues, and since $|g'(s)|\leq \alpha$
for all $s\in\mathbb{R}$,
we may choose $\alpha > 0$ such that
$
\mu_-(D_2f(\tau,x)), \mu_+(D_2f(\tau,x))
\in
\intcc{\lambda_-,\lambda_+}
$
for all $x\in X$. Thus,
\ref{e:th:SufficientConvexityConditionC11NormODE:Ellipsoid:L+-} and
\ref{e:th:SufficientConvexityConditionC11NormODE:Ellipsoid:M2} hold.\\
Direct calculation shows that for $x=0$, $h=(0,1)$, $x_0 = (-r,0)$,
and $t_1 > 0$, the left hand side of
\ref{e:th:ConvexityCriterionC2ODE:Ellipsoid:1} equals
$r M_2 K(M_1-\varepsilon)$, and thus, application of Theorem
\ref{th:ConvexityCriterionC2ODE:Ellipsoid} shows
$\varphi(t_1,0,\cBall(x_0,r))$ is not convex.
If $t_1 < 0$, set $\nu_{\pm} = \lambda_{\mp} \pm \varepsilon/3$ and
$x_0 = (r,0)$ instead to proceed in exactly the same way.
As $K(M_1) \not= 0$ implies $t_1 \not= t_0$,
the claim is proved after applying the change of coordinates
$x \mapsto x - x_0$.
\end{proof}

Recently, Azhmyakov, Flockerzi and Raisch
have investigated the closed-loop variant
\begin{subequations}
\label{e:ControlSystem:ClosedLoop}
\begin{align}
\label{e:ControlSystem:ClosedLoop:ODE}
\dot x &= f( x, u(x) )\\
\label{e:ControlSystem:ClosedLoop:ICandConstraint}
x(0) &= x_0, u(x) \in U%\\
\end{align}
\end{subequations}
of the control system \ref{e:ControlSystem},
where the control set $U \subseteq \mathbb{R}^m$ is
compact and convex, solutions to \ref{e:ControlSystem:ClosedLoop} are
assumed to exist on $[0,t_1]$ for any $U$-valued measurable feedback
$u$ and to be uniformly bounded, the right hand side
$f\colon \mathbb{R}^n \times \mathbb{R}^m \to \mathbb{R}^n$
is merely Lipschitz-continuous, and admissible controls are $U$-valued
and Lipschitz-continuous \cite{AzhmyakovFlockerziRaisch07}.
It has been claimed that the reachable set
\[
\mathcal{R}(t_1,x_0)
=
\Menge{x(t_1)}{ \text{$x$ is a solution of
    \ref{e:ControlSystem:ClosedLoop} for some admissible feedback $u$}}
\]
is convex provided $t_1$ is small enough \cite{AzhmyakovFlockerziRaisch07}.
It follows from geometric ideas employed in
\cite{Plis74,Lojasiewicz78,Polyak04} and the present paper, in
particular from Example \ref{ex:CounterExUniformlyConvexBSs} and
Corollary
\ref{cor:th:SufficientConvexityConditionC11NormODE:Ellipsoid:sharpness},
that such a claim must be wrong. We provide a simple counterexample below.

\begin{example}
\label{ex:AzhmyakovFlockerziRaisch07:Countreexample1}
Consider the special case
\begin{align*}
\dot x_1 &= u(x),\\
\dot x_2 &= u(x)^2,\\
x(0) &= (0,0), u(x) \in [0,1]
\end{align*}
of \ref{e:ControlSystem:ClosedLoop} on some time interval
$[0,t_1]$. As the constant feedbacks $0$ and $1$ are
admissible, we have $(0,0), (t_1,t_1) \in \mathcal{R}(t_1,(0,0))$ for
any $t_1 > 0$. If $\mathcal{R}(t_1,(0,0))$ were convex, then
$(t_1/2,t_1/2) \in \mathcal{R}(t_1,(0,0))$, which implies
$\int_0^{t_1} u(x(\tau)) d\tau = t_1/2 = \int_0^{t_1} u(x(\tau))^2 d\tau$
for some admissible control $u$ and corresponding solution $x$.
Hence, the integral $\int_0^{t_1} u(x(\tau)) - u(x(\tau))^2 d\tau$
vanishes, and so does its continuous, nonnegative integrand. It
follows that either $u \circ x = 0$ or $u \circ x = 1$, which is a
contradiction. So, $\mathcal{R}(t_1,(0,0))$ is not convex, which
proves \cite[Theorems 2 and 3]{AzhmyakovFlockerziRaisch07} wrong. It
is easily seen that the same system is also a counterexample to
\cite[Theorem 1]{AzhmyakovFlockerziRaisch07}.
\end{example}

\section{Application}
\label{s:Application}
In this section, we demonstrate the application of the results from
section \ref{s:ConvexityOfReachableSets} to the equations of the
damped mathematical pendulum,
\begin{subequations}
\label{e:Pendulum}
\begin{align}
\label{e:Pendulum:a}
\dot x_1 & = x_2,\\
\label{e:Pendulum:b}
\dot x_2 & = -\omega^2 \sin(x_1) - 2 \gamma x_2,
\end{align}
\end{subequations}
where $\omega>0$ and $\gamma\geq 0$
\cite{SansoneConti64}.
Note that the investigation of the convexity of reachable sets of more
general systems, such as a cart-pole system with a piecewise constant
control, can be reduced to the autonomous system
\ref{e:Pendulum} \cite{i05Bosen}.

The results of sections \ref{s:ConvexityUnderDiffeomorphisms} and
\ref{s:ConvexityOfReachableSets} cover the case of images of
ellipsoids as we have allowed for arbitrary inner products.
In this section, we restrict ourselves to the case of images of
Euclidean balls for the sake of simplicity.

The above ODE is of the form
\begin{equation}
\label{e:ODEautonomous}
\dot x = f(x),
\end{equation}
where $f\colon\mathbb{R}^2\to\mathbb{R}^2$ is given by
\[
f(x) =
\begin{pmatrix}
x_2\\
-\omega^2 \sin(x_1) - 2 \gamma x_2
\end{pmatrix}.
\]

\noindent
We first demonstrate a straightforward application of Corollary
\ref{cor:th:SufficientConvexityConditionC11NormODE:Ellipsoid}:

\begin{theorem}
\label{th:SufficientConvexityConditionPendulum:EuclideanBall}
Assume $\omega>0$,
$\gamma\geq0$ and
$t_1 \not= 0$, let $M_1$ and $R$ be given by
\begin{align}
\label{e:th:SufficientConvexityConditionPendulum:EuclideanBall:M1}
M_1
& =
-\sign(t_1) \gamma + 3 \sqrt{\gamma^2+(1 + \omega^2)^2/4},\\
\label{e:th:SufficientConvexityConditionPendulum:EuclideanBall:R}
R
&=
\frac{M_1}{\omega^2(\exp(M_1 |t_1|)-1)},
\end{align}
and let $\varphi$ be the flow of the pendulum equation
\ref{e:Pendulum}.
Then the image of any ball with radius not exceeding
$R$ under the map $\varphi(t_1,\cdot)$ is convex.
\end{theorem}
\begin{proof}
From
\begin{equation}
\label{e:ODEautonomous:f'}
f'(x)=
\begin{pmatrix}
0 & 1\\
-\omega^2 \cos(x_1)& - 2 \gamma
\end{pmatrix}
\end{equation}
we obtain the minimum and maximum eigenvalues
$
\lambda_{\pm}
=
-\gamma
\pm
\sqrt{\gamma^2+(1 + \omega^2)^2/4}
$
of the symmetric part $\frac{1}{2} (f'(x)+f'(x)^\ast)$ of $f'(x)$. As
\begin{equation}
\label{e:ODEautonomous:f''}
f''(x)h^2
=
\begin{pmatrix}0\\\omega^2 h_1^2\sin(x_1)\end{pmatrix},
\end{equation}
Corollary
\ref{cor:th:SufficientConvexityConditionC11NormODE:Ellipsoid} is
applicable with
$M_1 \not= 0$ and
$M_2 = \omega^2$ if the ball is closed. The theorem is proved since
the image of an open ball under the map
$\varphi(t_1,\cdot)$ is the interior of the image of the closure of that
ball.
\end{proof}

\noindent
The following larger bound on the radius is obtained from Theorem
\ref{th:ConvexityCriterionC2ODE:Ellipsoid}.

\begin{theorem}
\label{th:SufficientConvexityConditionPendulum:EuclideanBall:b}
Let $\omega,\gamma,\kappa_{\pm},t_1,R\in\mathbb{R}$ with
$0\leq\gamma\leq \omega$,
$1\leq \omega$,
$\kappa_{\pm} = \sqrt{\omega^2 \pm \gamma^2}$,
$2 \kappa_{-} t_1 \leq \pi$,
$0 < t_1$,
\begin{equation}
\label{th:SufficientConvexityConditionPendulum:EuclideanBall:b:R}
R
=
\frac{6\omega \kappa_{+}}{\left(1+(\omega+\gamma)^2\right)^{3/2}
  \sinh(\kappa_{+} t_1)(\cosh(2\kappa_{+} t_1) + 5 - 10 \exp(-\omega))},
\end{equation}
and let $\varphi$ be the flow of the pendulum equation
\ref{e:Pendulum}.
Then the image of any ball with radius not exceeding
$R$ under the map $\varphi(t_1,\cdot)$ is convex.
\end{theorem}
\begin{proof}
According to Theorem \ref{th:ConvexityCriterionC2ODE:Ellipsoid} it
suffices to prove
\begin{equation}
\label{e:th:SufficientConvexityConditionPendulum:EuclideanBall:b:1}
\int_0^{t_1}
\innerProd{(-h_2, h_1)}{
D_2\varphi(\tau,x_0)^{-1}
f''( \varphi(\tau,x_0) ) (D_2\varphi(\tau,x_0) h)^2
}
d\tau
\leq
1/R
\end{equation}
for all $x_0, h\in\mathbb{R}^2$ with $\|h\| = 1$.
$D_2\varphi(\cdot,x_0)$ is an operator solution of the
variational equation $\dot x = f'(\varphi(t,x_0))x$
of \ref{e:Pendulum} along $\varphi(\cdot,x_0)$, and
$D_2\varphi(0,x_0) = \id$.
Hence,
by
\ref{e:ODEautonomous:f'},
\ref{e:ODEautonomous:f''},
Cramer's rule, and the formula of Abel--Liouville,
\ref{e:th:SufficientConvexityConditionPendulum:EuclideanBall:b:1}
reduces to
\begin{equation}
\label{e:th:SufficientConvexityConditionPendulum:EuclideanBall:b:2}
\omega^2
\int_0^{t_1}
  \e^{2\gamma\tau}
  \sin\left( \varphi(\tau,x_0)_1 \right)
  \left( D_2\varphi(\tau,x_0) h \right)_1^3
d\tau
\leq
1/R.
\end{equation}
We shall establish
\ref{e:th:SufficientConvexityConditionPendulum:EuclideanBall:b:2}
under the assumption $\gamma < \omega$. The result for
$\gamma = \omega$ then follows from a continuity argument.
\makeatletter\def\p@enumi#1{#1}\makeatother
\begin{asparaenum}[1.)]
\item
\label{i:th:SufficientConvexityConditionPendulum:EuclideanBall:b:2}
In order to estimate the integrand in
\ref{e:th:SufficientConvexityConditionPendulum:EuclideanBall:b:2} we
investigate initial value problems
\begin{subequations}
\label{e:th:SufficientConvexityConditionPendulum:EuclideanBall:b:VergleichsIVP}
\begin{align}
\label{e:th:SufficientConvexityConditionPendulum:EuclideanBall:b:VergleichsODE}
\dot x &= A_{\rho} x + \begin{pmatrix}0\\u(t)\end{pmatrix},\\
\label{e:th:SufficientConvexityConditionPendulum:EuclideanBall:b:VergleichsIC}
x(0) &= h
\end{align}
\end{subequations}
for continuous $u$ and $h\geq 0$, $h\not=0$, where
\[
A_{\rho} =
\begin{pmatrix}
0    &  1\\
\rho & -2 \gamma
\end{pmatrix}
\]
for any $\rho \in \mathbb{R}$.
For $\lambda(t) \defas \exp(A_{\rho} t)h$ we obtain
\begin{equation}
\label{e:th:SufficientConvexityConditionPendulum:EuclideanBall:b:SolPlus}
\lambda_1(t) =
\frac{1}{2 \kappa_{+}}\e^{-\gamma t}
\left(
  ( h_1 (\kappa_{+} + \gamma) + h_2 ) \e^{\kappa_{+} t}
+ ( h_1 (\kappa_{+} - \gamma) - h_2 ) \e^{-\kappa_{+} t}
\right)
\end{equation}
if $\rho = \omega^2$, and
$\lambda_1(t) =
\e^{-\gamma t}
\left(
h_1 \cos(\kappa_{-} t) + \kappa_{-}^{-1} (h_1 \gamma + h_2) \sin(\kappa_{-} t)
\right)$,
if $\rho = -\omega^2$.
In particular, $\lambda_1$ is positive on
$\intoo{0,t_1}$ in the latter case.
Moreover, for $\rho = \omega^2$,
\ref{e:th:SufficientConvexityConditionPendulum:EuclideanBall:b:VergleichsODE}
is cooperative as $A_{\omega^2}$ is essentially
nonnegative \cite{AngeliSontag03}.
% \cite{Kato82,Walter00,FarinaRinaldi00,AngeliSontag03}.
Hence, $x(t) \leq \lambda(t)$ for all $t \in [0,t_1]$
if $u$ is nonpositive on $[0,t_1]$.
\item
\label{i:th:SufficientConvexityConditionPendulum:EuclideanBall:b:3}
We claim
$
0
\leq
\left(D_2\varphi(t,x_0) h\right)_1
\leq
\left( \exp\left( A_{\omega^2} t \right) h \right)_1
$
for all $t \in [0,t_1]$ and all $x_0, h\in\mathbb{R}^2$ with $h\geq0$.
First note that $x \defas D_2\varphi(\cdot,x_0) h$ is
a solution of the initial value problem
\ref{e:th:SufficientConvexityConditionPendulum:EuclideanBall:b:VergleichsIVP}
with $\rho = -\omega^2$ and
$u(t) = \omega^2
\left( 1 - \cos\left( \varphi(t,x_0)_1 \right) \right)
x_1(t)$. This implies
\begin{equation}
\label{e:th:SufficientConvexityConditionPendulum:EuclideanBall:b:3}
x_1(t) - \left(\exp\left( A_{-\omega^2} t \right) h\right)_1
=
\int_0^t
\left( \exp\left( A_{-\omega^2} (t-\tau) \right) \right)_{1,2}
u(\tau)
d\tau
\end{equation}
for all $t\in \mathbb{R}$.
Let $h_1>0$, assume $x_1$ has a zero in
$\intoo{0,t_1}$, and let $s$ be the smallest such zero. Then, by step
\ref{i:th:SufficientConvexityConditionPendulum:EuclideanBall:b:2}, the
left hand side of
\ref{e:th:SufficientConvexityConditionPendulum:EuclideanBall:b:3} is
negative at $t=s$.
On the other hand, $u$ is nonnegative on $[0,s]$, which by step
\ref{i:th:SufficientConvexityConditionPendulum:EuclideanBall:b:2}
implies the integrand in
\ref{e:th:SufficientConvexityConditionPendulum:EuclideanBall:b:3} is
nonnegative
if $t = s$,
and hence, the right hand side
of \ref{e:th:SufficientConvexityConditionPendulum:EuclideanBall:b:3}
is nonnegative at $t=s$. This is a contradiction, so
$x_1$ is positive on $\intoo{0,t_1}$.\\
Observe now that $x$ is a solution of the initial value problem
\ref{e:th:SufficientConvexityConditionPendulum:EuclideanBall:b:VergleichsIVP},
this time with $\rho = \omega^2$ and
$u(t) = - \omega^2
\left( 1 + \cos\left( \varphi(t,x_0)_1 \right) \right)
x_1(t)$. As $u$ is nonpositive on $[0,t_1]$, application of step
\ref{i:th:SufficientConvexityConditionPendulum:EuclideanBall:b:2}
and a continuity argument (in $h_1$) completes the proof of the claim.
\item
\label{i:th:SufficientConvexityConditionPendulum:EuclideanBall:b:4}
From step
\ref{i:th:SufficientConvexityConditionPendulum:EuclideanBall:b:3},
\ref{e:th:SufficientConvexityConditionPendulum:EuclideanBall:b:SolPlus},
$1 \leq \omega$,
and 
$|\left(D_2\varphi(\tau,x_0) h\right)_1| \leq \|D_2\varphi(\tau,x_0)_{1,\cdot}\|$,
we obtain the bound
\begin{equation}
\label{e:th:SufficientConvexityConditionPendulum:EuclideanBall:b:*}
\e^{-\gamma \tau}
\left( 1 + (\kappa_{+} + \gamma)^2 \right)^{3/2}
\kappa_{+}^{-3}
\left( \cosh(\kappa_{+} \tau)^2 - \frac{1}{\omega^2 + 1} \right)^{3/2}
\end{equation}
for the modulus of the integrand in
\ref{e:th:SufficientConvexityConditionPendulum:EuclideanBall:b:2}.
We next show
\begin{equation}
\label{e:th:SufficientConvexityConditionPendulum:EuclideanBall:b:**}
g(x)
\defas
\left(1 - \frac{5}{3} x \exp(-\alpha)\right)
-
\left( 1 - x/(\alpha^2 + 1) \right)^{3/2}
\geq 0
\end{equation}
for all $\alpha \geq 1$ and all $x\in[0,1]$.
As $g'$ is monotone decreasing on $[0,1]$, the map $g$ is
concave. Therefore, since $g(0)=0$, it suffices to show that $g(1) \geq 0$,
i.e., that 
\begin{equation}
\label{e:th:SufficientConvexityConditionPendulum:EuclideanBall:b:***}
z(\alpha)
\defas
\frac{3}{5}
\exp(\alpha)
\left(
1-\frac{\alpha^3}{(\alpha^2+1)^{3/2}}
\right)
\geq 1
\end{equation}
for all $\alpha \geq 1$.
It is easy to see that $z'(\alpha)$ is a positive multiple of
$
(\alpha^2 + 1)^{5/2}
- \alpha^5 - \alpha^3 - 3 \alpha^2
$ and that $(\alpha^2 + 1)^5 - (\alpha^5 + \alpha^3 + 3 \alpha^2)^2$
is a polynomial in $s$ with nonnegative coefficients if $1+s$ is
substituted for $\alpha$. Hence $z'(\alpha) \geq 0$ for all
$\alpha \geq 1$. As $z(1) > 1$ is easily verified,
\ref{e:th:SufficientConvexityConditionPendulum:EuclideanBall:b:***},
and hence
\ref{e:th:SufficientConvexityConditionPendulum:EuclideanBall:b:**},
have been established for all $\alpha \geq 1$ and all $x\in[0,1]$.
From
\ref{e:th:SufficientConvexityConditionPendulum:EuclideanBall:b:*},
\ref{e:th:SufficientConvexityConditionPendulum:EuclideanBall:b:**} and
the fact that
$
\left( 1 + (\kappa_{+} + \gamma)^2 \right)^{3/2}
\kappa_{+}^{-3}
\leq
\left( 1 + (\omega + \gamma)^2 \right)^{3/2}
\omega^{-3}
$
we obtain the bound
\[
\left( 1 + (\omega + \gamma)^2 \right)^{3/2}
\omega^{-3}
\left(
\cosh(\kappa_{+} \tau)^3 - \frac{5}{3} \exp(-\omega) \cosh(\kappa_{+} \tau)
\right)
\]
for the modulus of the integrand in
\ref{e:th:SufficientConvexityConditionPendulum:EuclideanBall:b:2},
from which
\ref{e:th:SufficientConvexityConditionPendulum:EuclideanBall:b:2}
follows by integration.
\qedhere
\end{asparaenum}
\end{proof}

We would like to emphasize that the results we have established in this
section are of a global type, i.e., the bounds $R$ on the radius of a ball
in Theorems
\ref{th:SufficientConvexityConditionPendulum:EuclideanBall} and
\ref{th:SufficientConvexityConditionPendulum:EuclideanBall:b}
do not depend on the location of the ball in phase space. Instead,
convexity of reachable sets from arbitrary balls of initial states
is guaranteed, provided their radii do not exceed $R$.

As we have already remarked in section
\ref{s:ConvexityOfReachableSets}, the bound
\ref{e:th:SufficientConvexityConditionPendulum:EuclideanBall:R} on
radii is never smaller and in general larger than the bound from
\cite{Lojasiewicz78}. For the pendulum equations \ref{e:Pendulum} one
can show that these bounds coincide iff $\omega = 1$ and
$\gamma = 0$. In particular, in the undamped case, $\gamma = 0$,
\ref{e:th:SufficientConvexityConditionPendulum:EuclideanBall:M1} reads
$M_1 = \frac{3}{2}(1 + \omega^2)$, whereas \cite{Lojasiewicz78} would
give the bound
\ref{e:th:SufficientConvexityConditionPendulum:EuclideanBall:R} with
$M_1 = 3\max\{1,\omega^2\}$.
\ref{fig:bounds}(a) and \ref{fig:bounds}(b) show the bounds obtained from
Theorems \ref{th:SufficientConvexityConditionPendulum:EuclideanBall}
and \ref{th:SufficientConvexityConditionPendulum:EuclideanBall:b} and
from \cite{Lojasiewicz78} in comparison to a bound obtained
numerically for two sets of parameters, $\omega = 1$ and $\gamma = 0$,
and $\omega = 6.1$ and $\gamma = 0.2$. The latter parameters are
obtained from a model of an experimental cart-pole system when time is
measured in seconds \cite{Baldissera04}.
The results show that in both cases, the balls of initial states for
which Theorem
\ref{th:SufficientConvexityConditionPendulum:EuclideanBall:b} proves
the reachable sets are convex are large enough to be used in actual
computations, such as in local programming techniques \cite{Polyak01}
and polyhedral approximation of reachable sets discussed in section
\ref{s:intro}. For the second set of parameters, the bounds obtained
from Theorem
\ref{th:SufficientConvexityConditionPendulum:EuclideanBall}
and \cite{Lojasiewicz78} seem to be less useful.

\begin{figure}
\psfrag{T}[l][]{\scriptsize$t_1$}
\psfrag{lgr}[][b]{\scriptsize$\lg R$}
\psfrag{N}[r][]{\scriptsize$0$}
\psfrag{p01}[][]{\scriptsize$\phantom{-}1$}
\psfrag{1}[][]{\scriptsize$1$}
\psfrag{m01}[][]{\scriptsize$-1$}
\psfrag{14}[][]{\scriptsize$0.25$}
\psfrag{p12}[][]{$\phantom{-}\frac{1}{2}$}
\psfrag{m12}[][]{$-\frac{1}{2}$}
\psfrag{m32}[][]{$-\frac{3}{2}$}
\psfrag{n2}[][]{\scriptsize$\pi/2$}
\psfrag{m2}[][]{\scriptsize$-2$}
\psfrag{m6}[][]{\scriptsize$-6$}
\psfrag{15}[][]{\scriptsize$0.2$}
\psfrag{110}[][]{\scriptsize$0.1$}
\psfrag{150}[][]{\scriptsize$0.02$}
\psfrag{1100}[][]{\scriptsize$0.01$}
\includegraphics[width=.5\linewidth]{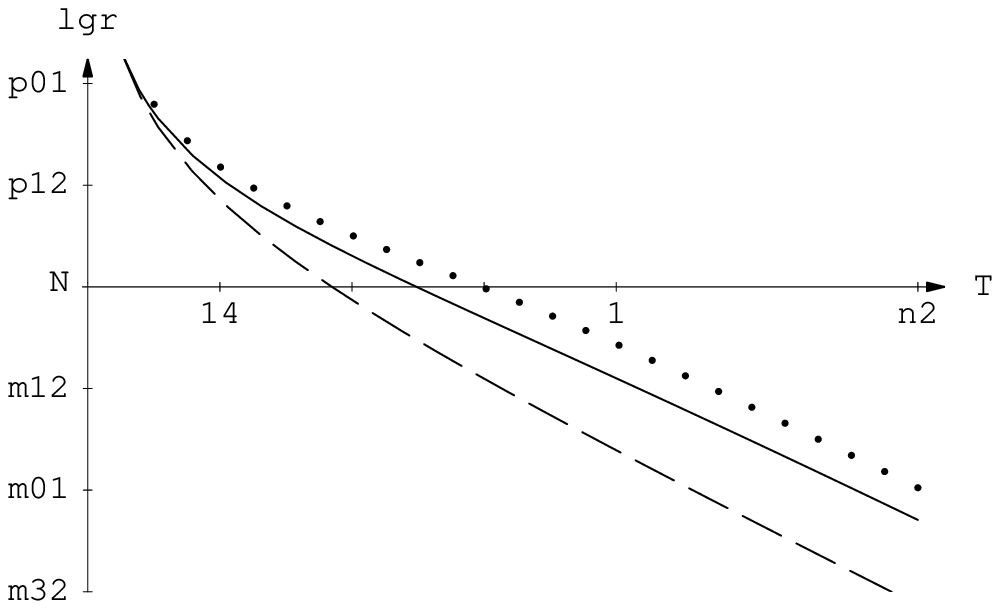}
\includegraphics[width=.5\linewidth]{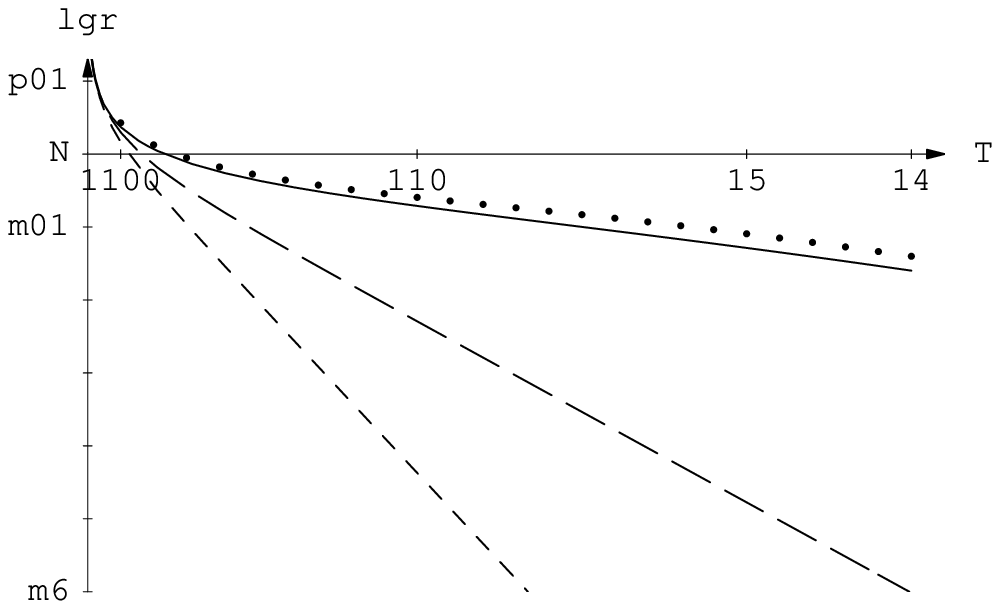}\\
\hspace*{\fill}(a)\hspace*{\fill}\hspace*{\fill}(b)\hspace*{\fill}
\caption{\label{fig:bounds}
Bound $R$ on the radius of $\Omega$ that ensures convexity of the
reachable set $\varphi(t_1,\Omega)$ over $t_1$, where $\lg R$ denotes the
logarithm to base $10$ of $R$.
(a) $\omega = 1$, $\gamma = 0$.
(b) $\omega = 6.1$, $\gamma = 0.2$.
($\bullet\bullet$: numerical bound,
{--- ---}: bound from Theorem
\ref{th:SufficientConvexityConditionPendulum:EuclideanBall},
{---}: bound from Theorem
\ref{th:SufficientConvexityConditionPendulum:EuclideanBall:b},
{- -}: bound from \cite{Lojasiewicz78}.)
}
\end{figure}

\section{Conclusions}
We have presented a novel necessary and sufficient condition for the
image of a sublevel set under a diffeomorphism to be convex.
That result has been applied to reachable sets from a ball of
initial states through solutions of an ordinary differential
equation, which has resulted in a necessary and sufficient
condition for the convexity of those reachable sets.
We have also established an upper bound on the radius of the ball
of initial states that ensures the reachable set is convex.
That bound is sharp and can be directly obtained from the right hand
side of the differential equation.
In finite dimensions, our results cover the
case of ellipsoids of initial states.
We have also demonstrated by means of an
example that the balls of initial states that result in convex
reachable sets are large enough to be used in actual
computations, for example in local programming
techniques \cite{Polyak01} and polyhedral approximation of reachable
sets discussed in section \ref{s:intro}.
\section*{Acknowledgment}
The author thanks
D.~Ferus (Berlin),
G.~P.~Peters (Berlin),
A.~Reibiger (Dresden)
and an anonymous referee
for their valuable hints, comments and discussions.
%\appendix
%%
%\dots
%%%%%%%%%%%%%%%%%%%%%%%%%%%%%%%%%%%%%%%%%%%%%%%%%%%%%%%%%%%%%%%%%%%%%%%%
\bibliographystyle{unsrtabb}
\bibliography{mrabbrev,03}
%%%%%%%%%%%%%%%%%%%%%%%%%%%%%%%%%%%%%%%%%%%%%%%%%%%%%%%%%%%%%%%%%%%%%%%%
\end{document}